\documentclass[a4paper,12pt]{article}
\usepackage{amsmath}
\usepackage[margin=1.2in]{geometry}
\usepackage{amsthm}
\usepackage{amssymb}
\usepackage{framed}
\usepackage[inline]{enumitem}
\usepackage{xcolor}
\usepackage{bm} 
\usepackage[all]{xy}
\usepackage[utf8]{inputenc}
\usepackage{setspace}
\usepackage{currfile}
\usepackage{comment}
\usepackage{mathabx}
\usepackage{textcomp}
\usepackage{hyperref}
\usepackage{thmtools}
\usepackage{accents}
\usepackage[nameinlink]{cleveref}

\newcommand\starred[1]{\accentset{\star}{#1}}

\usepackage[backend=biber,style=alphabetic,sorting=nyt,firstinits=true,language=english,maxbibnames=10]{biblatex}

\hypersetup{
    colorlinks,
    linkcolor={blue!80!black},
    citecolor={blue!80!black},
    urlcolor={blue!80!black}
}

\renewbibmacro{in:}{%
  \ifentrytype{article}{}{\printtext{\bibstring{in}\intitlepunct}}}

\DeclareSourcemap{
  \maps[datatype=bibtex, overwrite]{
    \map{
      \step[fieldset=language, null]
       \step[fieldset=shorthand, null]
    }
  }
}

\bibliography{efree.bib}

\newcommand{\ud}{\mathrm{d}}

\DeclareMathOperator{\supp}{supp}
\DeclareMathOperator{\Hom}{Hom}

\DeclareMathOperator{\Div}{Div}
\DeclareMathOperator{\id}{id}

\newcommand{\csub}{\Subset}

\newcommand{\bN}{\mathbb{N}}
\DeclareMathOperator{\dif}{d \!}
\newcommand{\coleq}{\mathrel{\mathop:}=}

\newcommand{\cA}{\mathcal{A}}
\newcommand{\cB}{\mathcal{B}}
\newcommand{\cC}{\mathcal{C}}
\newcommand{\cI}{\mathcal{I}}
\newcommand{\cE}{\mathcal{E}}

\newcommand{\cP}{\mathcal{P}}

\newcommand{\bC}{\mathbb{C}}
\newcommand{\bR}{\mathbb{R}}

\newcommand{\cD}{\mathcal{D}}

\newcommand{\cL}{\mathcal{L}}
\newcommand{\cM}{\mathcal{M}}
\newcommand{\cN}{\mathcal{N}}
\newcommand{\cS}{\mathcal{S}}

\newcommand{\cG}{\mathcal{G}}
\newcommand{\cU}{\mathcal{U}}

\newcommand{\D}{\mathrm{D}}
\newcommand{\DSK}{\mathrm{D}^{\mathrm{SK}}}

\newcommand{\csn}{\mathrm{csn}}

\newcommand{\pd}{\partial}

\newcommand{\abso}[1]{\left|#1\right|}
\newcommand{\norm}[1]{\lVert#1\rVert}

\newcommand{\basice}{\cB^{\textrm{c}}}
\newcommand{\quotiente}{\cG^{\textrm{c}}}
\newcommand{\mode}{\cM^{\textrm{c}}}
\newcommand{\nege}{\cN^{\textrm{c}}}
\newcommand{\iotae}{\iota^{\textrm{c}}}
\newcommand{\sigmae}{\sigma^{\textrm{c}}}

\declaretheorem[name=Theorem,refname={theorem,theorems},Refname={Theorem,Theorems}]{theorem}
\declaretheorem[name=Definition,refname={definition,definitions},Refname={Definition,Definitions},sibling=theorem]{definition}
\declaretheorem[name=Proposition,refname={proposition,propositions},Refname={Proposition,Propositions},sibling=theorem]{proposition}

\declaretheorem[name=Lemma,refname={lemma,lemmas},Refname={Lemma,Lemmas},sibling=theorem]{lemma}
\declaretheorem[name=Remark,refname={remark,remarks},Refname={Remark,Remarks},sibling=theorem]{remark}

\newcommand{\e}{\varepsilon}
\title{Colombeau algebras without asymptotics}
\author{Eduard A.~Nigsch\footnote{Wolfgang Pauli Institute, Oskar-Morgenstern-Platz 1, 1090 Vienna, Austria.\newline email:\href{mailto:eduard.nigsch@univie.ac.at}{eduard.nigsch@univie.ac.at}}}

\begin{document}

\maketitle

\begin{abstract}
We present a construction of algebras of generalized functions of Colombeau-type which, instead of asymptotic estimates with respect to a regularization parameter, employs only topological estimates on certain spaces of kernels for its definition.
\end{abstract}




\section{Introduction}

Colombeau algebras, as introduced by J.~F.~Colombeau \cite{ColNew, ColElem}, today represent the most widely studied approach to embedding the space of Schwartz distributions into an algebra of generalized functions such that the product of smooth functions as well as partial derivatives of distributions are preserved. These algebras have found numerous applications in situations involving singular objects, differentiation and nonlinear operations (see, e.g., \cite{MOBook, GKOS, Nedeljkov}).

All constructions of Colombeau algebras so far incorporate certain \emph{asymptotic estimates} for the definition of the spaces of moderate and negligible functions, the quotient of which constitutes the algebra. There is a certain degree of freedom in the asymptotic scale employed for these estimates; while commonly a polynomial scale is used, generalizations in several directions are possible. For an overview we refer to works on asymptotic scales \cite{0919.46027,zbMATH01417251}, $(\cC, \cE, \cP)$-algebras \cite{zbMATH05590471}, sequences spaces with exponent weights \cite{1126.46030} and asymptotic gauges \cite{zbMATH06552177}.

In this article we will present an algebra of generalized functions which instead of asymptotic estimates employs only topological estimates on certain spaces of kernels for its definition. This is a direct generalization of the usual seminorm estimates valid for distributions.

We will first develop the most general setting in the local scalar case, namely that of diffeomorphism invariant full Colombeau algebras. We will then derive a simpler variant, similar to Colombeau's elementary algebra. Finally, we give canonical mappings into the most important Colombeau algebras, which points to a certain universality of the construction offered here.

\section{Preliminaries}

$\bN$ and $\bN_0$ denote the sets of positive and non-negative integers, respectively, and $\bR^+$ the set of nonnegative real numbers. Concerning distribution theory we use the notation and terminology of L.~Schwartz \cite{TD}.

Given any subsets $K,L \subseteq \bR^n$ (with $n \in \bN$) the relation $K \csub L$ means that $K$ is compact and contained in the interior $L^\circ$ of $L$.

Let $\Omega \subseteq \bR^n$ be open. $C^\infty(\Omega)$ is the space of complex-valued smooth functions on $\Omega$. For any $K,L \csub \Omega$, $m,l \in \bN_0$ and any bounded subset $B \subseteq C^\infty(\Omega)$ we set
\begin{alignat*}{2}
\norm{f}_{K,m} & \coleq \sup_{x \in K, \abso{\alpha} \le m} \abso{\pd^\alpha f(x)} & \qquad & (f \in C^\infty(\Omega)), \\
\norm{\vec\varphi}_{K,m; L, l} & \coleq \sup_{\substack{x \in K, \abso{\alpha} \le m \\ y \in L, \abso{\beta} \le l}} \abso{ \pd_x^\alpha \pd_y^\beta \vec\varphi(x)(y) } & \qquad & (\vec\varphi \in C^\infty(\Omega, \cD(\Omega))), \\
\norm{\vec\varphi}_{K,m; B} & \coleq \sup_{\substack{x \in K, \abso{\alpha} \le m \\ f \in B}} \abso{ \langle f(y), \pd_x^\alpha \vec\varphi(x)(y) \rangle } & \qquad & (\vec\varphi \in C^\infty(\Omega, \cE'(\Omega))).
\end{alignat*}
Note that $\norm{\cdot}_{K,m}$, $\norm{\cdot}_{K,m; L, l}$ and $\norm{\cdot}_{K,m; B}$ are continuous seminorms on the respective spaces.

We define $\vec\delta \in C^\infty(\Omega, \cE'(\Omega))$ by $\vec\delta(x) \coleq \delta_x$ for $x \in \Omega$, where $\delta_x$ is the delta distribution at $x$.

$\cD_L(\Omega)$ is the space of test functions on $\Omega$ with support in $L$. For two locally convex spaces $E$ and $F$, $\cL(E,F)$ denotes the space of linear continuous mappings from $E$ to $F$, endowed with the topology of bounded convergence. By $\cU_x(\Omega)$ we denote the filter base of open neighborhoods of a point $x$ in $\Omega$, and by $\cU_K(\Omega)$ the filter base of open neighborhoods of $K$. By $\csn(E)$ we denote the set of continuous seminorms of a locally convex space $E$. $B_r(x) \coleq \{ y \in \bR^n : \norm{y-x} < r \}$ is the open Euclidean ball of radius $r>0$ at $x \in \bR^n$.

Our notion of smooth functions between arbitrary locally convex spaces is that of convenient calculus \cite{KM}. In particular, $\ud^k f$ denotes the $k$-th differential of a smooth mapping $f$.

\section{Construction of the algebra}\label{sec_quot}

Throughout this section let $\Omega \subseteq \bR^n$ be a fixed open set. Let $\cC$ be the category of locally convex spaces with smooth mappings in the sense of convenient calculus as morphisms.

\newcommand{\unterstrich}{{-}}

\begin{definition}Consider $C^\infty(\unterstrich, \cD(\Omega))$ and $C^\infty(\unterstrich)$ as sheaves with values in $\cC$. We define the \emph{basic space} of nonlinear generalized functions on $\Omega$ to be the set of sheaf homomorphisms
\[ \cB(\Omega) \coleq \Hom ( C^\infty(\unterstrich, \cD(\Omega)), C^\infty(\unterstrich)). \] 
\end{definition}
Hence, an element of $\cB(\Omega)$ is given by a family $(R_U)_U$ of mappings
\[ R_U \in C^\infty( C^\infty(U, \cD(\Omega)), C^\infty(U))\qquad (U \subseteq \Omega\ \textrm{open})\]
satisfying $R_U(\vec\varphi)|_V = R_V(\vec\varphi|_V)$ for all open subsets $V \subseteq U$ and $\vec\varphi \in C^\infty(U, \cD(\Omega))$. We will casually write $R$ in place of $R_U$.

\begin{remark}The basic space $\cB(\Omega)$ can be identified with the set of all mappings $R \in C^\infty( C^\infty(\Omega, \cD(\Omega)), C^\infty(\Omega))$ such that for any open subset $U \subseteq \Omega$ and $\vec\varphi, \vec\psi \in C^\infty(\Omega, \cD(\Omega))$ the equality $\vec\varphi|_U = \vec\psi|_U$ implies $R(\vec\varphi)|_U = R(\vec\psi)|_U$ (cf.~\cite{specfull}).
\end{remark}

$\cB(\Omega)$ is a $C^\infty(\Omega)$-module with multiplication
\[ (f \cdot R)_U(\vec\varphi) = f|_U \cdot R_U(\vec\varphi) \]
for $R \in \cB(\Omega)$, $f \in C^\infty(\Omega)$, $U \subseteq \Omega$ open and $\vec\varphi \in C^\infty(U, \cD(\Omega))$. Moreover, it is an associative commutative algebra with product $(R \cdot S)_U(\vec\varphi) \coleq R_U(\vec\varphi) \cdot S_U(\vec\varphi)$.

A distribution $u \in \cD'(\Omega)$ defines a sheaf morphism from $C^\infty(\unterstrich, \cD(\Omega))$ to $C^\infty(\unterstrich)$. In fact, for $U \subseteq \Omega$ open and $\vec\varphi \in C^\infty(U, \cD(\Omega))$ the function $x \mapsto \langle u, \varphi(x) \rangle$ is an element of $C^\infty(U)$ (see \cite[Chap.~IV, \S 1, Th.~II, p.~105]{TD} or \cite[Theorem 40.2, p.~416]{Treves}). More abstractly, this can be seen using the theory of topological tensor products \cite{FDVV,zbMATH03145498,Treves} as follows:
\[ C^\infty(U, \cD(\Omega)) \cong C^\infty(U) \widehat\otimes \cD(\Omega) \cong \cL ( \cD'(\Omega), C^\infty(U)), \]
where $C^\infty(U) \widehat\otimes \cD(\Omega)$ denotes the completed projective tensor product of $C^\infty(U)$ and $\cD(\Omega)$.
The assignment $\vec\varphi \mapsto \langle u, \vec\varphi \rangle$ is smooth, being linear and continuous \cite[1.3, p.~9]{KM}. Hence, we have the following embeddings of distributions and smooth functions into $\cB(\Omega)$:
\begin{definition}
 We define $\iota \colon \cD'(\Omega) \to \cB(\Omega)$ and $\sigma \colon C^\infty(\Omega) \to \cB(\Omega)$ by
\begin{alignat*}{2}
 (\iota u)(\vec\varphi)(x) & \coleq \langle u, \vec\varphi(x) \rangle & \qquad &(u \in \cD'(\Omega)) \\
 (\sigma f)(\vec\varphi)(x) & \coleq f(x) & \qquad &(f \in C^\infty(\Omega))
\end{alignat*}
for $\vec\varphi \in C^\infty(U, \cD(\Omega))$ with $U \subseteq \Omega$ open and $x \in U$.
\end{definition}
Clearly $\iota$ is linear and $\sigma$ is an algebra homomorphism. Directional derivatives on $\cB(\Omega)$ then are defined as follows:
\begin{definition}
Let $X \in C^\infty(\Omega, \bR^n)$ be a smooth vector field and $R \in \cB(\Omega)$. We define derivatives $\widetilde \D_X \colon \cB(\Omega) \to \cB(\Omega)$ and $\widehat \D_X \colon \cB(\Omega) \to \cB(\Omega)$
by
\begin{align*}
(\widetilde \D_X R)(\vec\varphi) & \coleq \D_X ( R_U ( \vec\varphi)) \\
(\widehat \D_X R)(\vec\varphi) & \coleq - \dif R_U ( \vec\varphi) (\DSK_X \vec\varphi) + \D_X ( R_U ( \vec\varphi))\\
\intertext{for $\vec\varphi \in C^\infty(U, \cD(\Omega))$ with $U \subseteq \Omega$ open, where we set}
\DSK_X \vec\varphi &\coleq \D_X \vec\varphi + \D_X^w \circ \vec\varphi.
\end{align*}
\end{definition}
Here, $(\D_X \vec\varphi)(x)$ is the directional derivative of $\vec\varphi$ at $x$ in direction $X(x)$ and $(\D_X^\omega \circ \vec\varphi)(x)$ is the Lie derivative of $\vec\varphi(x)$ considered as a differential form, given by $\D_X^\omega ( \vec\varphi(x)) = \D_X ( \vec\varphi(x)) + (\Div X)(x) \cdot \vec\varphi(x)$.

Note that both $\widetilde \D_X$ and $\widehat \D_X$ satisfy the Leibniz rule. We have
\[ \widetilde \D_x \circ \sigma = \sigma \circ \widetilde \D_X, \quad \widehat \D_X \circ \sigma = \sigma \circ \widehat \D_X, \quad \widehat\D_X \circ \iota = \iota \circ \widehat \D_X. \]
While $\widetilde \D_X$ is $C^\infty(\Omega)$-linear in $X$, $\widehat \D_X$ is only $\bC$-linear in $X$. We refer to \cite{papernew, bigone} for a discussion of the role of these derivatives in differential geometry.

\newcommand{\poly}{\cP}
\newcommand{\poli}{\cI}

\begin{definition}
For $k \in \bN_0$ we set
\begin{align*}
 \poly_k & \coleq \bR^+ [y_0, \dotsc, y_k], \\
 \poli_k & \coleq \{ \lambda \in \bR^+ [y_0, \dotsc, y_k, z_0, \dotsc, z_k]\ |\ \lambda(y_0, \dotsc, y_k, 0, \dotsc, 0) = 0 \}.
\end{align*}
\end{definition}

More explicitly, $\poly_k$ is the commutative semiring of polynomials in the $k+1$ commuting variables $y_0, \dotsc, y_k$ with coefficients in $\bR^+$. Similarly, $\poli_k$ is the commutative semiring in the $2(k+1)$ commuting variables $y_0, \dotsc, y_k, z_0, \dotsc, z_k$ with coefficients in $\bR^+$ and such that, if $\lambda \in \poli_k$ is given by the finite sum
\[ \lambda = \sum_{\alpha,\beta \in \bN_0^{k+1}} \lambda_{\alpha \beta} y^\alpha z^\beta, \]
then $\lambda_{\alpha 0} = 0$ for all $\alpha$. Note that $\poly_k$ is a subsemiring of $\poly_{k+1}$ and $\poli_k$ a subsemiring of $\poli_{k+1}$. Furthermore, $\poli_k$ is an ideal in $\poly_k$ if $\poly_k$ is considered as a subsemiring of $\bR^+ [y_0, \dotsc, y_k, z_0, \dotsc, z_k]$. Given $\lambda \in \poly_k$ and $y_i \le y_i'$ for $i=0 \dotsc k$ we have $\lambda(y) \le \lambda(y')$. For $\lambda, \mu \in \poly_k$ we write $\lambda \le \mu$ if $\lambda(y) \le \mu(y)$ for all $y \in (\bR^+)^{k+1}$, and similarly for $\lambda, \mu \in \poli_k$.

We can now formulate the following definitions of moderateness and negligibility, not involving any asymptotic estimates:

\begin{definition}\label{def_mod}
 An element $R \in \cB(\Omega)$ is called \emph{moderate} if
\begin{gather*}
(\forall x \in \Omega)\ (\exists U \in \cU_x(\Omega))\ (\forall K,L \csub U)\ (\forall m,k \in \bN_0)\\
(\exists c,l \in \bN_0)\ (\exists \lambda \in \poly_k) \ (\forall \vec\varphi_0,\dotsc,\vec\varphi_k \in C^\infty(U, \cD_L(U))):\\
\norm{ \ud^k R(\vec\varphi_0)(\vec\varphi_1,\dotsc,\vec\varphi_k)}_{K, m} \le \lambda ( \norm{\vec\varphi_0}_{K,c; L, l}, \dotsc, \norm{\vec\varphi_k}_{K,c; L, l}).
\end{gather*}
The subset of all moderate elements of $\cB(\Omega)$ is denoted by $\cM(\Omega)$.
\end{definition}

\begin{definition}\label{def_negl}
 An element $R \in \cB(\Omega)$ is called \emph{negligible} if
\begin{gather*}
(\forall x \in \Omega)\ (\exists U \in \cU_x(\Omega))\ (\forall K,L \csub U)\ (\forall m,k \in \bN_0)\ (\exists c,l \in \bN_0)\\
(\exists \lambda \in \poli_k)\ (\exists B \subseteq C^\infty(\Omega)\ \textrm{bounded})\ (\forall \vec\varphi_0, \dotsc, \vec\varphi_k \in C^\infty(U, \cD_L(U))):\\
\norm{\ud^kR(\vec\varphi_0)(\vec\varphi_1,\dotsc,\vec\varphi_k)}_{K, m} \\
\le \lambda ( \norm{\vec\varphi_0}_{K,c; L, l}, \dotsc, \norm{\vec\varphi_k}_{K,c; L, l}, \norm{ \vec\varphi_0 - \vec\delta}_{K, c; B}, \norm{ \vec\varphi_1}_{K, c; B}, \dotsc, \norm{ \vec\varphi_k}_{K, c; B}).
\end{gather*}
The subset of all negligible elements of $\cB(\Omega)$ is denoted by $\cN(\Omega)$.
\end{definition}

It is worthwile to discuss possible simplifications of these definitions, which at this stage should be considered more as a proof of concept than as the definite form they should have. First, we note that we cannot replace $(\forall x \in \Omega)\ (\exists U \in \cU_x(\Omega))\ (\forall K,L \csub U)$ by $(\forall K,L \csub \Omega)$. In fact, in the second case $K$ and $L$ can be distant from each other, while in the first case it suffices to control the situation where $K$ and $L$ are close to each other. However, the following result shows that we can always assume $K \csub L$ and that the $\vec\varphi_0,\dotsc,\vec\varphi_k$ are given merely on an arbitrary open neighborhood of $K$, i.e., as elements of the direct limit $C^\infty(K, \cD_L(\Omega)) \coleq \varinjlim_{V \in \cU_K(\Omega)} C^\infty(V, \cD_L(\Omega))$:
\begin{proposition}\label{prop_simpl}
 Let $R \in \cB(\Omega)$. Then $R$ is moderate if and only if
\begin{gather*}
(\forall x \in \Omega)\ (\exists U \in \cU_x(\Omega))\ (\forall K,L \csub U: K \csub L)\ (\forall m,k \in \bN_0)\\
(\exists c,l \in \bN_0)\ (\exists \lambda \in \poly_k) \ (\forall \vec\varphi_0,\dotsc,\vec\varphi_k \in C^\infty(K, \cD_L(U))):\\
\norm{ \ud^k R(\vec\varphi_0)(\vec\varphi_1,\dotsc,\vec\varphi_k)}_{K, m} \le \lambda ( \norm{\vec\varphi_0}_{K,c; L, l}, \dotsc, \norm{\vec\varphi_k}_{K,c; L, l}).
\end{gather*}
Similarly, $R$ is negligible if and only if
\begin{gather*}
(\forall x \in \Omega)\ (\exists U \in \cU_x(\Omega))\ (\forall K,L \csub U: K \csub L)\ (\forall m,k \in \bN_0)\ (\exists c,l \in \bN_0)\\
(\exists \lambda \in \poli_k)\ (\exists B \subseteq C^\infty(U)\ \textrm{bounded})\ (\forall \vec\varphi_0, \dotsc, \vec\varphi_k \in C^\infty(K, \cD_L(U))):\\
\norm{\ud^kR(\vec\varphi_0)(\vec\varphi_1,\dotsc,\vec\varphi_k)}_{K, m} \\
\le \lambda ( \norm{\vec\varphi_0}_{K,c; L, l}, \dotsc, \norm{\vec\varphi_k}_{K,c; L, l}, \norm{ \vec\varphi_0 - \vec\delta}_{K, c; B}, \norm{ \vec\varphi_1}_{K, c; B}, \dotsc, \norm{ \vec\varphi_k}_{K, c; B}).
\end{gather*}
\end{proposition}
\begin{proof}
Obviously each of these conditions is weaker than the corresponding one of \Cref{def_mod} or \Cref{def_negl}.

Suppose we are given $R \in \cB(\Omega)$ such that the condition stated for moderateness holds. Given $x \in \Omega$ there hence exists some $U \in \cU_x(\Omega)$. Now given arbitrary $K, L \csub U$ we choose a set $L' \csub U$ such that $K \cup L \csub L'$. Fixing $m, k \in \bN_0$ for the moderateness test, for $(K,L')$ we hence obtain $c,l \in \bN_0$ and $\lambda \in \poly_k$. Now fix some $\vec\varphi_0, \dotsc, \vec\varphi_k \in C^\infty(U, \cD_L(U))$; each of those represents an element of $C^\infty(K, \cD_{L'}(U))$, whence we have the estimate
\begin{align*}
 \norm{ \ud^k R(\vec\varphi_0)(\vec\varphi_1, \dotsc, \vec\varphi_k)}_{K,m} & \le \lambda ( \norm{\vec\varphi_0}_{K,c; L', l}, \dotsc, \norm{\vec\varphi_k}_{K,c; L', l}) \\
& = \lambda ( \norm{\vec\varphi_0}_{K,c; L, l}, \dotsc, \norm{\vec\varphi_k}_{K,c; L, l}).
\end{align*}
where the last equality follows because the $\vec\varphi_0,\dotsc,\vec\varphi_k$ take values in $\cD_L(U)$. This shows that $R$ is moderate.

For the case of negligibility we proceed similarly until we obtain $c,l \in \bN_0$, $\lambda \in \poli_k$ and $B \subseteq C^\infty(U)$. Let $\chi \in \cD(U)$ be such that $\chi \equiv 1$ on a neighborhood of $L'$ and set $B' \coleq \{ \chi f\ |\ f \in B \} \subseteq C^\infty(\Omega)$, which is bounded. For any $\vec\varphi_0, \dotsc, \vec\varphi_k$ we then obtain
\begin{gather*}
 \norm{ \ud^k R(\vec\varphi_0)(\vec\varphi_1, \dotsc, \vec\varphi_k)}_{K,m} \le \\
\le \lambda ( \norm{\vec\varphi_0}_{K,c; L', l}, \dotsc, \norm{\vec\varphi_k}_{K,c; L', l}, \norm{ \vec\varphi_0 - \vec\delta}_{K, c; B}, \norm{ \vec\varphi_1}_{K, c; B}, \dotsc, \norm{ \vec\varphi_k}_{K, c; B}) \\
= \lambda ( \norm{\vec\varphi_0}_{K,c; L, l}, \dotsc, \norm{\vec\varphi_k}_{K,c; L, l}, \norm{ \vec\varphi_0 - \vec\delta}_{K, c; B'}, \norm{ \vec\varphi_1}_{K, c; B'}, \dotsc, \norm{ \vec\varphi_k}_{K, c; B'})  
\end{gather*}
which proves negligibility of $R$.
\end{proof}

If the test of \Cref{def_mod}, \Cref{def_negl} or \Cref{prop_simpl} holds on some $U$ then clearly it also holds on any open subset of $U$. The following characterization of moderateness and negligiblity is obtained by applying polarization identities to the differentials of $R$:

\begin{lemma}\label{modchar}Let $R \in \cB(\Omega)$.
 \begin{enumerate}[label=(\roman*)]
  \item \label{modchar.1} $R$ is moderate if and only if
\begin{gather*}
(\forall x \in \Omega)\ (\exists U \in \cU_x(\Omega))\ (\forall K,L \csub U)\ (\forall m,k \in \bN_0)\\
(\exists c,l \in \bN_0)\ (\exists \lambda \in \poly_{\min(1,k)}) \ (\forall \vec\varphi, \vec\psi \in C^\infty(U, \cD_L(U))):\\
\norm{ \ud^k R(\vec\varphi)(\vec\psi,\dotsc,\vec\psi)}_{K, m} \le \left\{
\begin{aligned}
& \lambda( \norm{\vec\varphi}_{K,c; L, l} ) & & \textrm{if }k = 0, \\
& \lambda( \norm{\vec\varphi}_{K,c; L, l}, \norm{\vec\psi}_{K,c; L, l} ) & & \textrm{if }k \ge 1.
\end{aligned}
\right.
\end{gather*}
 \item \label{modchar.2} $R$ is negligible if and only if
\begin{gather*}
(\forall x \in \Omega)\ (\exists U \in \cU_x(\Omega))\ (\forall K,L \csub U)\ (\forall m,k \in \bN_0)\ (\exists c,l \in \bN_0)\\
(\exists \lambda \in \poli_{\min(1,k)})\ (\exists B \subseteq C^\infty(\Omega)\ \textrm{bounded})\ (\forall \vec\varphi, \vec\psi \in C^\infty(U, \cD_L(U))):\\
\norm{\ud^kR(\vec\varphi)(\vec\psi,\dotsc,\vec\psi)}_{K, m} \\
\le \left\{
\begin{aligned}
& \lambda( \norm{\vec\varphi}_{K,c; L, l}, \norm{ \vec\varphi - \vec\delta}_{K, c; B} ) & & \textrm{if }k = 0, \\
& \lambda( \norm{\vec\varphi}_{K,c; L, l}, \norm{\vec\psi}_{K,c; L, l}, \norm{ \vec\varphi - \vec\delta}_{K, c; B}, \norm{ \vec\psi}_{K, c; B}  ) & & \textrm{if }k \ge 1.
\end{aligned}
\right.
\end{gather*}
\end{enumerate}
\end{lemma}
\begin{proof}We assume $k \ge 1$, as for $k=0$ the statements are identical.

\ref{modchar.1} ``$\Rightarrow$'': One obtains $\lambda \in \poly_k$ such that
\begin{align*}
 \norm{\ud^k R(\vec\varphi)(\vec \psi, \dotsc, \vec \psi)}_{K,m} & \le \lambda ( \norm{\vec\varphi}_{K,c; L, l}, \norm{\vec\psi}_{K,c; L, l}, \dotsc, \norm{\vec\psi}_{K,c; L ,l } ) \\
& = \lambda' ( \norm{\vec\varphi}_{K,c;L,l}, \norm{\vec\psi}_{K,c;L, l})
\end{align*}
with $\lambda' \in \poly_1$ given by $\lambda'(y_0, y_1) = \lambda(y_0, y_1, \dotsc, y_1)$.

``$\Leftarrow$'': One obtains $\lambda \in \poly_1$. We then use the polarization identity \cite[eq.\ (7), p.~471]{zbMATH06298343}
\[
 \ud^k R (\vec\varphi_0)(\vec\varphi_1, \dotsc, \vec\varphi_k) = \frac{1}{n!} \sum_{a=1}^k (-1)^{k-a} \sum_{\substack{J \subseteq \{1 \dotsc k\}\\\abso{J} = a}} \Delta^*(\ud^k R(\vec\varphi_0))(S_J)
\]
where $S_J \coleq \sum_{i \in J} \vec\varphi_i$ and we have set $\Delta^*(\ud^k R ( \vec\varphi_0))(\vec\psi) = \ud^k R (\vec\varphi_0)(\vec\psi, \dotsc, \vec\psi)$. Hence,
\begin{gather*}
 \norm{\ud^k R(\vec\varphi_0)(\vec\varphi_1, \dotsc, \vec\varphi_k)}_{K,m} \le \frac{1}{n!} \sum_{a=1}^k \sum_{\abso{J} = a} \norm{\Delta^*(\ud^k R(\vec\varphi_0))(S_J)}_{K,m} \\
 \le \frac{1}{n!} \sum_{a=1}^k \sum_{\abso{J} = a} \lambda ( \norm{\vec\varphi_0}_{K,c; L, l}, \norm{S_J}_{K,c;L,l}) \\
\le \frac{1}{n!} \sum_{a=1}^k \sum_{\abso{J} = a} \lambda ( \norm{\vec\varphi_0}_{K,c;L,l}, \sum_{i \in J}\norm{\vec\varphi_i}_{K,c;L,l}) \\
= \lambda' ( \norm{\vec\varphi_0}_{K,c;L,l}, \dotsc, \norm{\vec\varphi_k}_{K,c;L,l})
\end{gather*}
with $\lambda' \in \poly_k$ given by
\[ \lambda'(y_0, \dotsc, y_k) = \frac{1}{n!} \sum_{a=1}^k \sum_{\abso{J} = a} \lambda ( y_0, \sum_{i \in J} y_i ). \]

\ref{modchar.2} ``$\Rightarrow$'': We have $\lambda \in \poli_k$ such that
\begin{gather*}
 \norm{\ud^k R(\vec\varphi)(\vec\psi, \dotsc, \vec \psi)}_{K,m} \le \lambda ( \norm{\vec\varphi}_{K,c; L, l}, \norm{\vec\psi}_{K,c;L,l}, \dotsc, \norm{\vec\psi}_{K,c;L,l},\\
\norm{\vec\varphi - \vec\delta}_{K,c;B},\norm{\vec\psi}_{K,c;B}, \dotsc, \norm{\vec\psi}_{K,c;B}) \\
= \lambda' ( \norm{\vec\varphi}_{K,c;L;l}, \norm{\vec\psi}_{K,c;L,l}, \norm{\vec\varphi - \vec\delta}_{K,c;B}, \norm{\vec\psi}_{K,c;B})
\end{gather*}
with $\lambda' \in \poli_k$ given by
\[ \lambda' ( y_0, y_1, z_0, z_1) = \lambda(y_0, y_1, \dotsc, y_1, z_0, z_1, \dotsc, z_1). \]

``$\Leftarrow$'': We obtain $\lambda \in \poli_1$ such that, as above,
\begin{gather*}
 \norm{\ud^k R(\vec\varphi_0)(\vec\varphi_1, \dotsc, \vec\varphi_k)}_{K,m} \\
\le \frac{1}{n!} \sum_{a=1}^k \sum_{\abso{J}=a} \lambda ( \norm{\vec\varphi_0}_{K,c;L,l}, \norm{S_J}_{K,c;L,l}, \norm{\vec\varphi_0 - \vec\delta}_{K,c; B}, \norm{S_J}_{K,c; B}) \\
\le \frac{1}{n!} \sum_{a=1}^k \sum_{\abso{J}=a} \lambda ( \norm{\vec\varphi_0}_{K,c;L,l}, \sum_{i \in J} \norm{\vec\varphi_i}_{K,c;L,l}, \norm{\vec\varphi_0 - \vec\delta}_{K,c;B}, \sum_{i \in J}\norm{\vec\varphi_i}_{K,c;B}) \\
= \lambda' ( \norm{\vec\varphi_0}_{K,c;L,l}, \dotsc, \norm{\vec\varphi_k}_{K,c;L,l}, \norm{\vec\varphi_0 - \vec\delta}_{K,c;B}, \norm{\vec\varphi_1}_{K,c;B}, \dotsc, \norm{\vec\varphi_k}_{K,c;B})
\end{gather*}
with $\lambda' \in \poli_k$ given by
\[ \lambda'(y_0, \dotsc, y_k, z_0, \dotsc, z_k) = \frac{1}{n!} \sum_{a=1}^k \sum_{\abso{J}=a} \lambda(y_0, \sum_{i \in J}y_i, z_0, \sum_{i \in J}z_i). \qedhere \]
\end{proof}
Note that the polarization identities could be applied also in the formulation of \Cref{prop_simpl}.

\begin{proposition}\label{negmod}$\cN(\Omega) \subseteq \cM(\Omega)$.
\end{proposition}
\begin{proof}
Let $R \in \cN(\Omega)$ and fix $x \in \Omega$ for the moderateness test. By negligibility of $R$ there exists $U \in \cU_x(\Omega)$ as in \Cref{def_negl}. Let $K,L \csub U$ and $m,k \in \bN_0$ be arbitrary. Then there exist $c,l,\lambda$ and $B$ such that the estimate of \Cref{def_negl} holds. We know that $\lambda \in \poli_k$ is given by a finite sum
\[ \lambda(y_0, \dotsc, y_k, z_0, \dotsc, z_k) = \sum_{\alpha, \beta} \lambda_{\alpha\beta} y^\alpha z^\beta. \]
It suffices to show that there are $\lambda_1, \lambda_2 \in \poly_0$ such that for any $\vec\varphi \in C^\infty(U, \cD_L(U))$ we have the estimates
\begin{align}
 \label{douno}\norm{\vec\varphi - \vec\delta}_{K,c; B} & \le \lambda_1 ( \norm{\vec\varphi}_{K,c; L, l} ), \\
 \label{dodue}\norm{\vec\varphi}_{K,c; B} & \le \lambda_2 ( \norm{\vec\varphi}_{K,c; L, l} ).
\end{align}
In fact, these inequalities imply
\begin{align*}
 \norm{\ud^k R (\vec\varphi_0)&(\vec\varphi_1, \dotsc, \vec\varphi_k)}_{K,m} \\
 & \le \sum_{\alpha,\beta} \lambda_{\alpha\beta} \norm{\vec\varphi_0}^{\alpha_0}_{K,c; L, l} \cdot \dotsc \cdot \norm{\vec\varphi_k}^{\alpha_k}_{K,c; L, l} \\
 & \quad \cdot \norm{\vec\varphi_0 - \vec\delta}^{\beta_0}_{K,c; B} \cdot \norm{\vec\varphi_1}^{\beta_1}_{K,c; B} \cdot \dotsc \cdot \norm{\vec\varphi_k}^{\beta_k}_{K,c; B} \\
& \le \sum_{\alpha,\beta} \lambda_{\alpha\beta} \norm{\vec\varphi_0}^{\alpha_0}_{K,c; L, l} \cdot \dotsc \cdot \norm{\vec\varphi_k}^{\alpha_k}_{K,c; L, l} \\
&\quad \cdot \lambda_1 ( \norm{\vec\varphi_0}_{K,c; L, l} ) ^{\beta_0} \cdot \lambda_2 ( \norm{\vec\varphi_1}_{K,c; L, l} )^{\beta_1} \dotsm \lambda_2 ( \norm{\vec\varphi_k}_{K,c; L, l})^{\beta_k} \\
& = \lambda' ( \norm{\vec\varphi_0}_{K,c; L, l}, \dotsc, \norm{\vec\varphi_k}_{K,c; L, l} )
\end{align*}
with $\lambda' \in \poly_k$ given by
\[ \lambda'(y_0, \dotsc, y_k) = \sum \lambda_{\alpha\beta} y^\alpha \lambda_1(y_0)^{\beta_0} \lambda_2(y_1)^{\beta_1} \dotsm \lambda_2(y_k)^{\beta_k}.\]
Inequality \eqref{douno} is seen as follows:
\begin{align*}
\norm{\vec\varphi - \vec\delta}_{K,c; B} &= \sup_{\substack{x \in K, \abso{\alpha} \le c\\f \in B}} \abso{ \int_L f(y)\pd_x^\alpha\vec\varphi(x)(y) \dif y - \pd^\alpha f(x) } \\
& \le \abso{L} \cdot \sup_{f \in B} \norm{f}_{L,0} \cdot \norm{\vec\varphi}_{K,c; L, l} + \sup_{f \in B} \norm{f}_{K,c}\\
& = \lambda_1 ( \norm{\vec\varphi}_{K,c; L, l})
\end{align*}
with $\lambda_1(y_0) = \abso{L} \cdot \sup_{f \in B} \norm{f}_{L,0} \cdot y_0 + \sup_{f \in B} \norm{f}_{K,c}$, where $\abso{L}$ denotes the Lebesgue measure of $L$.
Similarly, inequality \eqref{dodue} results from
\begin{align*}
 \norm{\vec\varphi}_{K,c; B} & = \sup_{\substack{x \in K, \abso{\alpha} \le c\\f \in B}} \abso{ \int_L f(y) \pd_x^\alpha \vec\varphi(x)(y) \dif y } \\
 & \le \abso{L} \cdot \sup_{f \in B} \norm{f}_{L,0} \cdot \norm{\vec\varphi}_{K,c; L, l} \\
 &= \lambda_2 ( \norm{\vec\varphi}_{K,c; L, l})
\end{align*}
with $\lambda_2(y_0) = \abso{L} \cdot \sup_{f \in B} \norm{f}_{L,0} \cdot y_0$.
\end{proof}

\begin{proposition}\label{11}$\cM(\Omega)$ is a subalgebra of $\cB(\Omega)$ and $\cN(\Omega)$ is an ideal in $\cM(\Omega)$.
\end{proposition}
\begin{proof}
This is evident from the definitions.
\end{proof}

\begin{theorem}\label{T1}Let $u \in \cD'(\Omega)$ and $f \in C^\infty(\Omega)$. Then
\begin{enumerate}[label=(\roman*)]
\item\label{T1.1} $\iota u$ is moderate,
\item\label{T1.2} $\sigma f$ is moderate,
\item\label{T1.3} $\iota f - \sigma f$ is negligible, and
\item\label{T1.4} if $\iota u$ is negligible then $u=0$.
\end{enumerate}
\end{theorem}
\begin{proof}
\ref{T1.1}: Fix $x$ for the moderateness test and let $U \in \cU_x(\Omega)$ be arbitrary. Fix any $K,L \csub U$ and $m \in \bN_0$. Then there are constants $C = C(L) \in \bR^+$ and $l = l(L) \in \bN_0$ such that $\abso{\langle u, \varphi\rangle} \le C \norm {\varphi}_{L,l}$ for all $\varphi \in \cD_L(\Omega)$. Hence, we see that
\begin{gather*}
 \norm{(\iota u)(\vec\varphi_0)}_{K,m} = \norm { \langle u, \vec\varphi_0 \rangle }_{K,m} = \sup_{x \in K, \abso{\alpha} \le m} \abso{ \langle u, \pd_x^\alpha \vec\varphi_0(x) \rangle} \\
\le C \cdot \sup_{\substack{x \in K, \abso{\alpha} \le m\\y \in L, \abso{\beta} \le l}} \abso{ \pd_x^\alpha\pd_y^\beta \vec\varphi_0(x)(y)} = C \norm{\vec\varphi_0}_{K,m; L, l} = \lambda ( \norm{\vec\varphi_0}_{K,m; L, l}).
\end{gather*}
with $\lambda ( y_0 ) = C y_0$. Moreover, we have
\[ \norm{ \ud ( \iota u )(\vec\varphi_0)(\vec\varphi_1)}_{K,m} \le C \norm{\vec\varphi_1}_{K,m; L, l} = \lambda ( \norm{\vec\varphi_0}_{K,m; L, l}, \norm{\vec\varphi_1}_{K,m; L, l}) \]
with $\lambda(y_0, y_1) = C y_1$. Higher differentials of $\iota u$ vanish and the moderateness test is satisfied with $\lambda = 0$ for $k \ge 2$.

\ref{T1.2}: Fix $x$ and let $U \in \cU_x(\Omega)$ be arbitrary. For any $K,L \csub U$ and $m \in \bN_0$ we have
 \[ \norm{ ( \sigma f)(\vec\varphi_0) }_{K,m} = \norm{f}_{K, m} = \lambda(\norm{\vec\varphi_0}_{K, 0; L, 0}) \]
with $\lambda(y_0) = \norm{f}_{K,m}$. Differentials of $\sigma f$ vanish, i.e., $\lambda = 0$ for $k\ge 1$.

\ref{T1.3}: Fix $x$ and let $U \in \cU_x(\Omega)$ be arbitrary. For any $K,L \csub U$ and $m,k \in \bN_0$ we have
\begin{align*}
 (\iota f - \sigma f)(\vec\varphi_0) &= \langle f, \vec\varphi_0 - \vec\delta \rangle, \\
 \ud ( \iota f - \sigma f)(\vec\varphi_0)(\vec\varphi_1) &= \langle f, \vec\varphi_1 \rangle, \\
 \ud^k ( \iota f - \sigma f)(\vec\varphi_0)(\vec\varphi_1, \dotsc, \vec\varphi_k) &= 0 \quad \textrm{ for }k\ge 2.
\end{align*}
Hence, with $c = m$, $l = 0$ and $B = \{ f \}$ the negligibility test is satisfied with $\lambda(y_0, z_0) = z_0$ for $k=0$, $\lambda(y_0, y_1, z_0, z_1) = z_1$ for $k=1$ and $\lambda=0$ for $k \ge 2$.

\ref{T1.4}: We show that every point $x \in \Omega$ has an open neighborhood $V$ such that $u|_V = 0$, which implies $u=0$.

Given $x \in \Omega$, let $U \in \cU_x(\Omega)$ be as in the characterization of negligibility in \Cref{prop_simpl}. Choose an open neighborhood $V$ of $x$ such that $K \coleq \overline{V} \csub U$ and $r>0$ such that $L \coleq \overline{B_r(K)} \csub U$. With $k=m=0$, \Cref{prop_simpl} gives $c,l \in \bN_0$, $\lambda \in \poli_0$ and $B \subseteq C^\infty(U)$, where $\lambda$ has the form
\[ \lambda(y,z) = \sum_{\alpha \in \bN_0^n, \beta \in \bN} \lambda_{\alpha\beta} y^\alpha z^\beta. \]
Choose $\varphi \in \cD(\bR^n)$ with $\supp \varphi \subseteq B_1(0)$, $\int \varphi(x) \,\ud x = 1$ and $\int x^\gamma \varphi(x) \,\ud x = 0$ for $\gamma \in \bN_0^n$ with $0 < \abso{\gamma} \le q$, where $q$ is chosen such that $\beta(q+1) > \alpha(n+c+l)$ for all $\alpha,\beta$ with $\lambda_{\alpha\beta} \ne 0$ (e.g., take $q = (n+c+l) \deg_y \lambda$, where $\deg_y \lambda$ is the degree of $\lambda$ with respect to $y$). For $\e>0$ set $\varphi_\e(y) = \e^{-n} \varphi(y/\e)$. Then for $\e<r$, $\vec\varphi_\e(x)(y) \coleq \varphi_\e(y-x)$ defines an element $\vec\varphi_\e \in C^\infty(K, \cD_L(\Omega))$ because $\supp \varphi_\e(.-x) = x + \supp \varphi_\e \subseteq B_\e(x) \subseteq B_r(K) \subseteq L$ for $x \in B_{r-\e}(K)$. Consequently, we have
\[ \norm{(\iota u)(\vec\varphi_\e)}_{K,0} \le \lambda ( \norm{\vec\varphi_\e}_{K,c; L, l}, \norm{\vec\varphi_\e - \vec\delta}_{K,c; B}). \]
Because of the estimates
\begin{align*}
\norm{\vec\varphi_\e}_{K,c; L, l} & = O(\e^{-(n+l+c)}) \\
\norm{\vec\varphi_\e - \vec\delta}_{K,c;B} & = O(\e^{q+1}),
\end{align*}
which may be verified by a direct calculation, we have
\[ \norm{(\iota u)(\vec\varphi_\e)}_{K,0} \le \sum_{\alpha, \beta} \lambda_{\alpha, \beta} \cdot O(\e^{-\alpha(n+c+l)}) \cdot O(\e^{\beta(q+1)}) \to 0 \]
by the choice of $q$, which means that $(\iota u)(\vec\varphi_\e)|_V \to 0$ in $C(V)$ and hence also in $\cD'(V)$. On the other hand, we have
\[ \langle u, \vec\varphi_\e \rangle|_V \to u|_V \]
in $\cD'(V)$, as is easily verified. This completes the proof.
\end{proof}

\begin{theorem}For $X \in C^\infty(\Omega, \bR^n)$ we have
 \begin{enumerate}[label=(\roman*)]
  \item \label{T4} $\widetilde \D_X ( \cM(\Omega)) \subseteq \cM(\Omega)$ and $\widehat \D_X ( \cM(\Omega)) \subseteq \cM(\Omega)$,
  \item \label{T5} $\widetilde \D_X ( \cN(\Omega)) \subseteq \cN(\Omega)$ and $\widehat \D_X ( \cN(\Omega)) \subseteq \cN(\Omega)$.
 \end{enumerate}
\end{theorem}
\begin{proof}
The claims for $\widetilde \D_X$ are clear because
\begin{gather*}
 \norm{\ud^k (\widetilde\D_X R)(\vec\varphi)(\vec \psi, \dotsc, \vec\psi)}_{K,m} = \norm{\D_X(\ud^k R (\vec\varphi)(\vec \psi, \dotsc, \vec\psi))}_{K,m} \\
 \le C \norm{\ud^k R(\vec\varphi)(\vec\psi, \dotsc, \vec \psi)}_{K,m+1}
\end{gather*}
for some constant $C$ depending on $X$. As to $\widehat \D_X$, we have to deal with terms of the form
\[ \ud^{k+1}R(\vec\varphi)(\DSK_X\vec\varphi, \vec\psi, \dotsc, \vec \psi)\quad\textrm{and}\quad \ud^kR(\vec\varphi)(\DSK_X\vec\psi, \vec\psi, \dotsc, \vec\psi) \]
for which we use the estimate
\[ \norm{\DSK_X\vec\varphi}_{K,c; L,l} \le C \norm{\vec\varphi}_{K,c,+1; L, l+1} \]
for some constant $C$ depending on $X$.
\end{proof}

We now come to the quotient algebra.

\begin{definition}We define the Colombeau algebra of generalized functions on $\Omega$ by $\cG(\Omega) \coleq \cM(\Omega) / \cN(\Omega)$.
\end{definition}

$\cG(\Omega)$ is a $C^\infty(\Omega)$-module and an associative commutative algebra with unit $\sigma(1)$. $\iota$ is a linear embedding of $\cD'(\Omega)$ and $\sigma$ an algebra embedding of $C^\infty(\Omega)$ into $\cG(\Omega)$ such that $\iota f = \sigma f$ in $\cG(\Omega)$ for all smooth functions $f \in C^\infty(\Omega)$. Furthermore, the derivatives $\widehat \D_X$ and $\widetilde \D_X$ are well-defined on $\cG(\Omega)$.

Finally, we establish sheaf properties of $\cG$. Note that for $\Omega' \csub \Omega$ open, the restriction $R|_{\Omega'}(\vec\varphi) \coleq R(\vec\varphi)$ is well-defined because for $U \subseteq \Omega'$ open we have $C^\infty(U, \cD(\Omega')) \subseteq C^\infty(U, \cD(\Omega))$.

\begin{proposition}Let $R \in \cB(\Omega)$ and $\Omega' \subseteq \Omega$ be open. If $R$ is moderate then $R|_{\Omega'}$ is moderate; if $R$ is negligible then $R|_{\Omega'}$ is negligible.
\end{proposition}

\begin{proof}
Suppose that $R \in \cM(\Omega)$. Fix $x \in \Omega'$, which gives $U \in \cU_x(\Omega)$. Set $U' \coleq U \cap \Omega' \in \cU_x(\Omega')$ and let $K,L \csub U'$ and $m,k \in \bN_0$ be arbitrary. Then there are $c,l,\lambda$ as in \Cref{def_mod}. Let now $\vec\varphi_0', \dotsc, \vec\varphi_k' \in C^\infty(U', \cD_L(U'))$ be given. Choose $\rho \in \cD(U')$ such that $\rho \equiv 1$ on a neighborhood of $K$. Then $\rho \cdot \vec\varphi'_i \in C^\infty(U, \cD_L(U))$ ($i=0 \dotsc k$) and
\begin{align*}
 \norm{\ud^k R|_{\Omega'} (\vec\varphi_0')(\vec\varphi_1',&\dotsc,\vec\varphi_k')}_{K, m} = \norm{\ud^k R|_{\Omega'} ( \rho \vec\varphi_0')(\rho \vec\varphi'_1, \dotsc, \rho \vec\varphi'_k)}_{K, m} \\
 &= \norm{\ud^k R(\rho \vec\varphi'_0)(\rho \vec\varphi'_1,\dotsc,\rho\vec\varphi'_k)}_{K,m} \\
 &\le \lambda ( \norm{ \rho \vec\varphi_0'}_{K,c; L, l}, \dotsc, \norm{\rho \vec\varphi_k'}_{K,c; L, l}) \\
&= \lambda ( \norm{ \vec\varphi_0'}_{K,c; L, l}, \dotsc, \norm{\vec\varphi_k'}_{K,c; L, l}).
\end{align*}
Hence, the moderateness test is satisfied for $R|_{\Omega'}$.

Now suppose that $R \in \cN(\Omega)$. For the negligibility test fix $x \in \Omega'$, which gives $U \in \cU_x(\Omega)$. Set $U' \coleq U \cap \Omega'$ and let $K,L \csub U'$ and $m,k \in \bN_0$ be arbitrary. Then $\exists c,l,B,\lambda$ as in \Cref{def_negl}. Let now $\vec\varphi_0', \dotsc, \vec\varphi_k' \in C^\infty(U', \cD_L(U'))$ be given. Choose $\rho \in \cD(U')$ such that $\rho \equiv 1$ on a neighborhood of $K$. Then $\rho \cdot \vec\varphi'_i \in C^\infty(U, \cD_L(U))$ ($i=0 \dotsc k$) and
\begin{gather*}
 \norm{\ud^k R|_{\Omega'} (\vec\varphi_0')(\vec\varphi_1',\dotsc,\vec\varphi_k')}_{K, m} = \norm{\ud^k R|_{\Omega'} ( \rho \vec\varphi_0')(\rho \vec\varphi'_1, \dotsc, \rho \vec\varphi'_k)}_{K, m} \\
 = \norm{\ud^k R(\rho \vec\varphi'_0)(\rho \vec\varphi'_1,\dotsc,\rho\vec\varphi'_k)}_{K,m} \\
 \le \lambda ( \norm{ \rho \vec\varphi_0'}_{K,c; L, l}, \dotsc, \norm{\rho \vec\varphi_k'}_{K,c; L, l}, \norm{\rho \vec\varphi_0' - \vec\delta}_{K, c; B}, \dotsc, \norm{\rho\vec\varphi_k'}_{K,c; B}) \\
= \lambda ( \norm{ \vec\varphi_0'}_{K,c; L, l}, \dotsc, \norm{\vec\varphi_k'}_{K,c; L, l}, \norm{\vec\varphi_0' - \vec\delta}_{K,c; B}, \dotsc, \norm{\vec\varphi_k'}_{K,c; B})
\end{gather*}
which shows negligibility of $R|_{\Omega'}$.
\end{proof}

\begin{proposition}$\cG(\unterstrich)$ is a sheaf of algebras on $\Omega$.
\end{proposition}
\begin{proof}
Let $X \subseteq \Omega$ be open and $(X_i)_i$ be a family of open subsets of $\Omega$ such that $\bigcup_i X_i = X$.

We first remark that if $R \in \cB(X)$ satisfies $R|_{X_i} \in \cN(X_i)$ for all $i$ then $R \in \cN(X)$, as is evident from the definition of negligibility.

Suppose now that we are given $R_i \in \cM(X_i)$ such that $R_i|_{X_i \cap X_j} - R_j|_{X_i \cap X_j} \in \cN(X_i \cap X_j)$ for all $i,j$ with $X_i \cap X_j \ne \emptyset$. Let $(\chi_i)_i$ be a partition of unity subordinate to $(X_i)_i$, i.e., a family of mappings $\chi_i \in C^\infty(X)$ such that $0 \le \chi_i \le 1$, $(\supp \chi_i)_i$ is locally finite, $\sum_i \chi_i(x) = 1$ for all $x \in X$ and $\supp \chi_i \subseteq X_i$. Choose functions $\rho_i \in C^\infty(X_i, \cD(X_i))$ which are equal to $1$ on an open neighborhood of the diagonal in $X_i \times X_i$ for each $i$. For $V \subseteq X$ open and $\vec\varphi \in C^\infty(V, \cD(X))$ we define $R_V(\vec\varphi) \in C^\infty(V)$ by
\begin{equation}\label{jaenner}
R_V(\vec\varphi) \coleq \sum_i \chi_i|_V \cdot (R_i)_{V \cap X_i} ( \rho_i|_{V \cap X_i} \cdot \vec\varphi|_{V \cap X_i}).
\end{equation}


For showing smoothness of $R_V$ consider a curve $c \in C^\infty(\bR, C^\infty(V, \cD(X)))$. We have to show that $t \mapsto R_V(c(t))$ is an element of $C^\infty(\bR, C^\infty(V))$. By \cite[3.8, p.~28]{KM} it suffices to show that for each open subset $W \subseteq V$ which is relatively compact in $V$ the curve $t \mapsto R_V(c(t))|_W = R_W(c(t)|_W)$ is smooth, but this holds because the sum in \eqref{jaenner} then is finite. Hence, $(R_V)_V \in \cB(\Omega)$.

Fix $x \in X$ for the moderateness test. There is a finite index set $F$ and an open neighborhood $W \in \cU_x(X)$ such that $W \cap \supp \chi_i \ne \emptyset$ implies $i \in F$. We can also assume that $x \in \bigcap_{i \in F}X_i$. Let $Y$ be a neighborhood of $x$ such that $\rho_i \equiv 1$ on $Y \times Y$ for all $i \in F$. For each $i \in F$ let $U_i \in \cU_x(X_i)$ be obtained from moderateness of $R_i$ as in \Cref{def_mod}. Set $U \coleq \bigcap_{i \in F} U_i \cap W \cap Y \in \cU_x(X)$, and let $K,L \csub U$ as well as $m,k \in \bN_0$ be arbitrary. For each $i \in F$ there are $c_i, l_i,\lambda_i$ such that for any $\vec\varphi_0, \dotsc, \vec\varphi_k \in C^\infty(U, \cD_L(U))$ we have
\[ \norm{\ud^k R_i ( \vec\varphi_0)(\vec\varphi_1, \dotsc, \vec\varphi_k)}_{K,m} \le \lambda_i ( \norm{\vec\varphi_0}_{K,c_i; L, l_i}, \dotsc, \norm{\vec\varphi_k}_{K,c_i; L, l_i}). \]
Now we have, for $\vec\varphi \in C^\infty(U, \cD_L(U))$,
\[ R(\vec\varphi)|_W = \sum_{i \in F} \chi_i|_W \cdot (R_i)_{W \cap X_i} ( \rho_i\vec\varphi|_{W \cap X_i}) \]
and hence, for $\vec\varphi_0, \dotsc, \vec\varphi_k \in C^\infty(U, \cD_L(U))$,
\begin{gather*}
 \ud^k R ( \vec\varphi_0)(\vec\varphi_1, \dotsc, \vec\varphi_k)|_W \\
 = \sum_{i \in F} \chi_i|_W \cdot \ud^k ((R_i)_{W \cap X_i}) ( \rho_i \vec\varphi_0|_{W \cap X_i})(\rho_i \vec\varphi_1|_{W \cap X_i}, \dotsc, \rho_i \vec\varphi_k|_{W \cap X_i}).
\end{gather*}
We see that
\begin{gather*}
 \norm{\ud^k R(\vec\varphi_0)(\vec\varphi_1, \dotsc, \vec\varphi_k)}_{K,m} \\
 \le \sum_{i \in F} C(m) \cdot \norm{\chi_i}_{K,m} \cdot \lambda_i ( \norm{\vec\varphi_0}_{K,c_i; L, l_i}, \dotsc, \norm{\vec\varphi_k}_{K,c_i; L, l_i}) \\
 = \lambda ( \norm{\vec\varphi_0}_{K,c; L, l}, \dotsc, \norm{\vec\varphi_k}_{K,c; L, l})
\end{gather*}
with $c = \max_{j \in F} c_j$, $l = \max_{j\in F} l_j$, some constant $C(m)$ coming from the Leibniz rule, and $\lambda \in \poly_k$ given by
\[ \lambda = \sum_{i \in F} C(m) \norm{\chi_i}_{K,m} \cdot \lambda_i. \]
This shows that $R$ is moderate. Finally, we claim that $R|_{X_j} - R_j \in \cN(X_j)$ for all $j$. For this we first note that
\[ (R|_{X_j} - R_j)(\vec\varphi) = \sum_i \chi_i|_{X_j} \cdot ( R_i ( \rho_i\vec\varphi|_{X_i \cap X_j}) - R_j(\vec\varphi)) \]
for $\vec\varphi \in C^\infty(X_j, \cD(X_j))$. Again, for $x \in X_j$ there is a finite index set $F$ and an open neighborhood $W \in \cU_x(X)$ such that $W \cap \supp \chi_i \ne \emptyset$ implies $i \in F$, and we can assume that $x \in \bigcap_{i \in F}X_i$. Let $Y$ be a neighborhood of $x$ such that $\rho_i \equiv 1$ on $Y \times Y$ for all $i \in F$ and let $U_i \in \cU_x ( X_i \cap X_j)$ be given by the negligibility test of $R_i|_{X_i \cap X_j} - R_j|_{X_i \cap X_j}$ according to \Cref{def_negl}. Set $U \coleq \bigcap_{i \in F} U_i \cap W \cap Y$. Fix any $K,L \csub U$ and $m,k \in \bN_0$. For each $i \in F$ there are $c_i, l_i, \lambda_i, B_i$ such that for $\vec\varphi_0, \dotsc, \vec\varphi_k \in C^\infty(U, \cD_L(U))$ we have
\begin{gather*}
\norm{\ud^k ( R_i|_{X_i \cap X_j} - R_j|_{X_i \cap X_j})(\vec\varphi_0)(\vec\varphi_1,\dotsc,\vec\varphi_k)}_{K,m} \\
\le \lambda_i ( \norm{\vec\varphi_0}_{K,c_i; L, l_i}, \dotsc, \norm{\vec\varphi_0 - \vec\delta}_{K, c_i; B_i}, \norm{\vec\varphi_1}_{K,c_i; B_i}, \dotsc, \norm{\vec\varphi_k}_{K,c_i; B_i}).
\end{gather*}
As above, we then have
\begin{gather*}
 \norm{\ud^k ( R|_{X_j} - R_j)(\vec\varphi_0)(\vec\varphi_1, \dotsc, \vec\varphi_k)}_{K,m} \\
 \le \sum_{i \in F} C(m) \cdot \norm{\chi_i}_{K,m} \cdot \lambda_i ( \norm{\vec\varphi_0}_{K,c_i; L, l_i}, \dotsc, \norm{\vec\varphi_0 - \vec\delta}_{K, c_i; B_i}, \norm{\vec\varphi_1}_{K,c_i; B_i}, \dotsc) \\
 \le \lambda ( \norm{\vec\varphi_0}_{K,c; L, l}, \dotsc, \norm{\vec\varphi_0 - \vec\delta}_{K, c; B}, \norm{\vec\varphi_0}_{K,c; B}, \dotsc )
\end{gather*}
with $c = \max_{i \in F} c_i$, $l = \max_{i \in F} l_i$, $B = \bigcup_{i\in F} B_i$, and $\lambda \in \poli_k$ given by
\[ \lambda = \sum_{i \in F} C(m) \norm{\chi}_{K,m} \cdot \lambda_i. \]
This completes the proof.
\end{proof}

\section{An elementary version}

We will now give a variant of the construction of \Cref{sec_quot} similar in spirit to Colombeau's elementary algebra \cite{ColElem}: if we only consider derivatives along the coordinate lines of $\bR^n$ we can replace the smoothing kernels $\vec\varphi \in C^\infty(U, \cD_L(\Omega))$ by convolutions. This way, one can use a simpler basic space which does not involve calculus on infinite dimensional locally convex spaces anymore:

\begin{definition}Let $\Omega \subseteq \bR^n$ be open. We set
\[ U(\Omega) \coleq \{ ( \varphi, x) \in \cD(\bR^n) \times \Omega\ |\ \supp \varphi + x \subseteq \Omega\}. \]
and define $\basice(\Omega)$ to be the set of all mappings $R \colon U(\Omega) \to \bC$ such that $R(\varphi, \cdot)$ is smooth for fixed $\varphi$.
\end{definition}

Note that this is almost the basic space used originally by Colombeau (see \cite[1.2.1, p.~18]{ColElem} or \cite[Definition 1.4.3, p.~59]{GKOS}) but with $\cD(\bR^n)$ in place of the space of test functions whose integral equals one. We now introduce a notation for the convolution kernel determined by a test function.
\begin{definition}For $\varphi \in \cD(\bR^n)$ we define $\starred\varphi \in C^\infty(\bR^n, \cD(\bR^n))$ by
 \[ \starred\varphi(x)(y) \coleq \varphi(y-x). \]
\end{definition}
In fact, with this definition we have $\langle u, \starred \varphi \rangle = u * \check \varphi$, where as usually we set $\check \varphi(y) \coleq \varphi(-y)$.
Furthermore, for $c \in \bN_0$ we write
\[ \norm{\varphi}_c \coleq \sup_{x \in \bR^n, \abso{\alpha} \le c} \abso{\pd^\alpha \varphi(x)} \qquad (\varphi \in \cD(\bR^n)). \]
The direct adaptation of \Cref{def_mod,def_negl} then looks as follows:

\begin{definition}\label{def_emodneg}
 Let $R \in \basice(\Omega)$. Then $R$ is called \emph{moderate} if
\begin{gather*}
(\forall x \in \Omega)\ (\exists U \in \cU_x(\Omega))\ (\forall K,L \csub U: K \csub L)\ (\forall m \in \bN_0)\\
(\exists c \in \bN_0)\ (\exists \lambda \in \poly_0) \ (\forall \varphi \in \cD(\bR^n): K + \supp \varphi \subseteq L):\\
\norm{ R(\varphi, .)}_{K, m} \le \lambda ( \norm{\varphi}_{c}).
\end{gather*}
The subset of all moderate elements of $\basice(\Omega)$ is denoted by $\mode(\Omega)$.

Similarly, $R$ is called \emph{negligible} if
\begin{gather*}
(\forall x \in \Omega)\ (\exists U \in \cU_x(\Omega))\ (\forall K,L \csub U: K \csub L)\ (\forall m \in \bN_0)\ (\exists c \in \bN_0)\\
(\exists \lambda \in \poli_0)\ (\exists B \subseteq C^\infty(U)\ \textrm{bounded})\ (\forall \varphi \in \cD(\bR^n): K + \supp \varphi \subseteq L):\\
\norm{R(\varphi, .)}_{K, m} \le \lambda ( \norm{\varphi}_{c}, \norm{ \starred\varphi - \vec\delta}_{K, c; B}).
\end{gather*}
The subset of all negligible elements of $\basice(\Omega)$ is denoted by $\nege(\Omega)$.
\end{definition}

It is convenient to work with the following simplification of these definitions.

\begin{proposition}\label{emodnegchar}$R \in \basice(\Omega)$ is moderate if and only if
 \begin{gather*}
  (\forall K \csub \Omega)\ (\exists r>0: \overline{B_r(K)} \csub \Omega)\ (\forall m \in \bN_0)\ (\exists c \in \bN_0)\\
  (\exists \lambda \in \poly_0) \ (\forall \varphi \in \cD(\bR^n): \supp \varphi \subseteq B_r(0)):\\
\norm{ R(\varphi, .)}_{K, m} \le \lambda ( \norm{\varphi}_{c}).
 \end{gather*}
Similarly, $R \in \basice(\Omega)$ is negligible if and only if
 \begin{gather*}
  (\forall K \csub \Omega)\ (\exists r>0: \overline{B_r(K)} \csub \Omega)\ (\forall m \in \bN_0)\ (\exists c \in \bN_0)\\
(\exists \lambda \in \poli_0)\ (\exists B \subseteq C^\infty(\Omega)\ \textrm{bounded})\ (\forall \varphi \in \cD(\bR^n): \supp \varphi \subseteq B_r(0)):\\
\norm{R(\varphi, .)}_{K, m} \le \lambda ( \norm{\varphi}_{c}, \norm{ \starred\varphi - \vec\delta}_{K, c; B}).
\end{gather*}
\end{proposition}

\begin{proof}
 Suppose $R$ is moderate and fix $K \csub \Omega$. We can cover $K$ by finitely many open sets $U_i$ obtained from \Cref{def_emodneg} and write $K = \bigcup_i K_i$ with $K_i \csub U_i$. Choose $r>0$ such that $L_i \coleq \overline{B_r(K_i)} \csub U_i$ for all $i$. Fixing $m$, by moderateness there exist $c_i$ and $\lambda_i$ for each $i$. Set $c = \max_i c_i$ and choose $\lambda$ with $\lambda \ge \lambda_i$ for all $i$. Now given $\varphi \in \cD(\bR^n)$ with $\supp \varphi \subseteq B_r(0)$ we also have $K_i + \supp \varphi \subseteq L_i$ and we can estimate
\begin{gather*}
 \norm{R(\varphi, .)}_{K,m} \le \sup_i \norm{R(\varphi,.)}_{K_i, m} \le \sup_i \lambda_i ( \norm{\varphi}_{c_i} ) \le \lambda ( \norm{\varphi}_c ).
\end{gather*}

Conversely, suppose the condition holds and fix $x \in \Omega$ for the moderateness test. Choose $a>0$ such that $\overline{B_a(x)} \csub \Omega$. By assumption there is $r>0$ with $\overline{B_{r+a}(x)} \csub \Omega$. Set $U \coleq B_{r/2}(x)$. Then, fix $K \csub L \csub U$ and $m$ for the moderateness test. There are $c$ and $\lambda$ by assumption. Now given $\varphi$ with $K + \supp \varphi \subseteq L$, we see that for $y \in \supp \varphi$ and an arbitrary point $z \in K$ we have $\abso{y} \le \abso{y+z-x} + \abso{z-x} < r$, which means that $\supp \varphi \subseteq B_r(0)$. But then $\norm{R(\varphi,.)}_{K,m} \le \lambda(\norm{\varphi}_c)$ as desired.

If $R$ is negligible we proceed similarly until the choice of $K_i \csub L_i \csub U_i$ and $m$ gives $c_i, \lambda_i$ and $B_i$. Choose $\chi_i \in \cD(U_i)$ with $\chi_i \equiv 1$ on a neighborhood of $L_i$, and define $B \coleq \bigcup_i \{ \chi_i f\ |\ f \in B_i \}$, which is bounded in $C^\infty(\Omega)$. Then with $c = \max_i c_i$ and $\lambda \ge \lambda_i$ for all $i$ we have
\[ \norm{R(\varphi, .)}_{K,m} \le \sup_i \lambda_i ( \norm{\varphi}_{c_i}, \norm{\starred\varphi - \vec\delta}_{K_i, c_i; B_i} ) \le \lambda ( \norm{\varphi}_c, \norm{\starred\varphi - \vec\delta}_{K,c; B}). \]

The converse is seen as for moderateness by restricting the elements of $B \subseteq C^\infty(\Omega)$ to $U$.
\end{proof}

The embeddings now take the following form.

\begin{definition}
 We define $\iotae \colon \cD'(\Omega) \to \basice(\Omega)$ and $\sigmae \colon C^\infty(\Omega) \to \basice(\Omega)$ by
\begin{alignat*}{2}
 (\iotae u)(\varphi, x) & \coleq \langle u, \varphi(.-x) \rangle & \qquad &(u \in \cD'(\Omega)) \\
 (\sigmae f)(\varphi, x) & \coleq f(x) & \qquad &(f \in C^\infty(\Omega)).
\end{alignat*}
\end{definition}

Partial derivatives on $\basice(\Omega)$ then can be defined via differentiation in the second variable:

\begin{definition}
Let $R \in \basice(\Omega)$. We define derivatives $\D_i \colon \basice(\Omega) \to \basice(\Omega)$ ($i=1, \dotsc, n$) by
\begin{align*}
(\D_i R)(\varphi, x) & \coleq \frac{\pd}{\pd x_i} ( x \mapsto R ( \varphi, x)).
\end{align*}
\end{definition}

\begin{theorem}We have $\D_i ( \mode(\Omega)) \subseteq \mode(\Omega)$ and $\D_i ( \nege(\Omega)) \subseteq \nege(\Omega)$,
 \end{theorem}
\begin{proof}
 This is evident from the definitions.
\end{proof}

\begin{proposition}We have $\D_i \circ \iota = \iota \circ \pd_i$ and $\D_i \circ \sigma = \sigma \circ \pd_i$.
\end{proposition}
\begin{proof}
$\D_i ( \iota u)(\varphi, x) = \frac{\pd}{\pd x_i} \langle u(y), \varphi(y-x) \rangle = \langle u(y), - (\pd_i \varphi)(y-x) \rangle = \langle \pd_i u(y), \varphi(y-x) \rangle = \iota( \pd_i u)(\varphi, x)$. The second claim is clear.
\end{proof}

\begin{proposition}$\nege(\Omega) \subseteq \mode(\Omega)$.
\end{proposition}
\begin{proof}
The result follows from
\[ \norm{\starred\varphi - \vec\delta}_{K,c; B} \le \lambda_1 ( \norm{\varphi}_{c_1} ) \]
for suitable $\lambda_1$ and $c_1$, which is seen as in the proof of \Cref{negmod}.
\end{proof}

Similarly to \Cref{11} we have:

\begin{proposition}$\mode(\Omega)$ is a subalgebra of $\basice(\Omega)$ and $\nege(\Omega)$ is an ideal in $\mode(\Omega)$.
\end{proposition}

\begin{theorem}\label{eT1}Let $u \in \cD'(\Omega)$ and $f \in C^\infty(\Omega)$. Then
\begin{enumerate}[label=(\roman*)]
\item\label{eT1.1} $\iotae u$ is moderate,
\item\label{eT1.2} $\sigmae f$ is moderate,
\item\label{eT1.3} $\iotae f - \sigmae f$ is negligible, and
\item\label{eT1.4} if $\iotae u$ is negligible then $u=0$.
\end{enumerate}
\end{theorem}

The proof is almost identical to that of \Cref{T1} and hence omitted.

\begin{definition}We define the elementary Colombeau algebra of generalized functions on $\Omega$ by $\quotiente(\Omega) \coleq \mode(\Omega) / \nege(\Omega)$.
\end{definition}

As before, one may show that $\quotiente$ is a sheaf.

%
%
%
%
%
%

\section{Canonical mappings}

In this section we show that the algebra $\cG$ constructed above is near to being universal in the sense that there exist canonical mappings from it into most of the classical Colombeau algebras which are compactible with the embeddings.


We begin by constructing a mapping $\cG(\Omega) \to \quotiente(\Omega)$.

\begin{definition}
Given $R \in \cB(\Omega)$ we define $\widetilde R \in \basice(\Omega)$ by
\[ \widetilde R(\varphi, x) \coleq R ( \vec\varphi)(x)\qquad ((\varphi,x) \in U(\Omega) ) \]
where $\vec\varphi \in C^\infty(\Omega, \cD(\Omega))$ is chosen such that $\vec\varphi = \starred \varphi$ in a neighborhood of $x$.
\end{definition}

This definition is meaningful: given $(\varphi,x)$ in $U(\Omega)$ we have $\supp \varphi(.-x') \subseteq \Omega$ for $x'$ in a neighborhood $V$ of $x$. Choosing $\rho \in \cD(\Omega)$ with $\supp \rho \subseteq V$ and $\rho \equiv 1$ in a neighborhood of $x$, we can take $\vec\varphi(x) \coleq \rho \starred \varphi$. Obviously, $\widetilde R(\varphi,x)$ does not depend on the choice of $\vec\varphi(x)$ and $\widetilde R(\varphi, .)$ is smooth, so indeed we have $\widetilde R \in \basice(\Omega)$.

\begin{proposition}\label{32}Let $R \in \cB(\Omega)$. Then the following holds:
 \begin{enumerate}[label=(\roman*)]
  \item\label{32.1} $\widetilde{\left.\iota u\right.} = \iotae u$ for $u \in \cD'(\Omega)$.
  \item\label{32.2} $\widetilde{\left.\sigma f\right.} = \sigmae f$ for $f \in C^\infty(\Omega)$.
  \item\label{32.3} $\widetilde R \in \mode(\Omega)$ for $R \in \cM(\Omega)$.
  \item\label{32.4} $\widetilde R \in \nege(\Omega)$ for $R \in \cN(\Omega)$.
 \end{enumerate}
\end{proposition}
\begin{proof}
\ref{32.1}: For $u \in \cD'(\Omega)$ we have
\[ \widetilde{\left.\iota u\right.}(\varphi,x) = (\iota u)(\vec\varphi)(x) = \langle u, \vec\varphi(x) \rangle = \langle u, \starred \varphi(x) \rangle = \langle u(y), \varphi(y-x) \rangle = (\iotae u)(\varphi, x). \]
\ref{32.2} is clear.

\ref{32.3}: Suppose that $R \in \cM(\Omega)$. Fixing $x \in \Omega$, we obtain $U$ as in \Cref{prop_simpl}. Let $K \csub L \csub U$ and $m$ be given, set $k=0$, and choose $L'$ such that $L \csub L' \csub U$. Then \Cref{prop_simpl} gives $c,l,\lambda$ such that for $\vec\varphi \in C^\infty(K, \cD_{L'}(U))$,
\[ \norm{R(\vec\varphi)}_{K,m} \le \lambda ( \norm{\vec\varphi}_{K,c; L', l} ). \]
Now for $\varphi \in \cD(\bR^n)$ with $K + \supp \varphi \subseteq L$ we have $\starred \varphi \in C^\infty(K, \cD_{L'}(U))$, which gives
\[ \norm{\widetilde R(\varphi,.)}_{K,m} = \norm{R(\starred \varphi)}_{K,m} \le \lambda ( \norm{ \starred \varphi }_{K,c; L', l}) \le \lambda ( \norm{\varphi}_{c+l} ) \]
which proves that $\widetilde R \in \mode(\Omega)$.

\ref{32.4}: Similarly, if $R \in \cN(\Omega)$ then for $x \in \Omega$ we have $U$ as in \Cref{prop_simpl}. For $K \csub L \csub U$, $m$ given, $k=0$, and $L'$ such that $L \csub L' \csub U$, we obtain $c,l,\lambda, B$ as in \Cref{prop_simpl} such that
\[ \norm{ R(\vec\varphi) }_{K,m} \le \lambda ( \norm{\vec\varphi}_{K,c; L', l}, \norm{\vec\varphi - \vec \delta}_{K,c; B} ) \]
and hence
\begin{align*}
 \norm{\widetilde R(\varphi, .)}_{K,m} & = \norm{ R ( \starred \varphi)}_{K,m} \\
& \le \lambda ( \norm{ \starred \varphi}_{K,c; L', l}, \norm{\starred \varphi - \vec\delta}_{K,c; B}) \\
&\le \lambda ( \norm { \varphi }_{c+l}, \norm{\starred \varphi - \vec\delta}_{K,c; B} ).
\end{align*}
which gives negligibility of $\widetilde R$.
\end{proof}

\subsection{The special algebra}

We define the special Colombeau algebra $\cG^s$ with the embedding as in \cite{zbMATH06172905}: fix a mollifier $\rho \in \cS(\bR^n)$ with
\[ \int \rho(x) \,\ud x = 1, \qquad \int x^\alpha \rho(x)\,\ud x = 0\qquad \forall \alpha \in \bN_0^n \setminus \{0\}. \]
Choosing $\chi \in \cD(\bR^n)$ with $0 \le \chi \le 1$, $\chi \equiv 1$ on $B_1(0)$ and $\supp \chi \subseteq B_2(0)$ we set
\[ \rho_\e(y) \coleq \e^{-n} \rho(y/\e),\quad \theta_\e(y) \coleq \rho_\e(y) \chi ( y \abso{\ln \e} ) \qquad (\e>0). \]
Moreover, with
\[ K_\e = \{ x \in \Omega\ |\ d(x, \bR^n \setminus \Omega) \ge \e \} \cap B_{1/\e}(0) \csub \Omega \qquad (\e>0) \]
we choose functions $\kappa_\e \in \cD(\Omega)$ such that $0 \le \kappa_\e \le 1$ and $\kappa_\e \equiv 1$ on $K_\e$. Then the special algebra $\cG^s(\Omega)$ is given by
\begin{align*}
 \cE^s(\Omega) &\coleq C^\infty(\Omega)^I\textrm{ with }I \coleq (0,1], \\
\cE^s_M(\Omega) &\coleq \{ (u_\e)_\e \in \cE^s(\Omega)\ |\ \forall K \csub \Omega\ \forall m\in \bN_0\ \exists N \in \bN: \norm{u_\e}_{K,m} = O(\e^{-N}) \}, \\
\cN^s(\Omega) &\coleq \{ (u_\e)_\e \in \cE^s(\Omega)\ |\ \forall K \csub \Omega\ \forall m\in \bN_0\ \forall N \in \bN: \norm{u_\e}_{K,m} = O(\e^N) \}, \\
\cG^s(\Omega) & \coleq \cE^s_M(\Omega) / \cN^s(\Omega), \\
(\iota^s u)_\e &\coleq \langle u, \vec\psi_\e \rangle \qquad (u \in \cD'(\Omega)), \\
(\sigma^s f)_\e &\coleq f \qquad \qquad (f \in C^\infty(\Omega)), \\
\vec\psi_\e(x)(y) &\coleq \theta_\e(x-y) \kappa_\e(y).
\end{align*}

\begin{definition}For $R \in \cB(\Omega)$ we define $R^s = (R^s_\e)_\e \in \cE^s(\Omega)$ by
 \[ R^s_\e (x) \coleq R ( \vec\psi_\e ) (x). \]
\end{definition}

\begin{proposition}\label{35}
 \begin{enumerate}[label=(\roman*)]
  \item\label{35.1} $(\iota u)^s = \iota^s u$ for $u \in \cD'(\Omega)$.
  \item\label{35.2} $(\sigma f)^s = \sigma^s f$ for $f \in C^\infty(\Omega)$.
  \item\label{35.3} $R^s \in \cE^s_M(\Omega)$ for $R \in \cM(\Omega)$.
  \item\label{35.4} $R^s \in \cN^s(\Omega)$ for $R \in \cN(\Omega)$.
 \end{enumerate}
\end{proposition}
\begin{proof}
\ref{35.1} and \ref{35.2} are clear.

For \ref{35.3} it suffices to show the needed estimate locally. Fix $x \in \Omega$, which gives $U \in \cU_x(\Omega)$ as in \Cref{prop_simpl}. Choose any $K,L$ such that $x \in K \csub L \csub U$, fix $m$, and set $k=0$. Then there are $c,l,\lambda$ as in \Cref{prop_simpl}. Because $\supp \vec\psi_\e(x) \subseteq B_{2 \abso{\ln \e}^{-1}}(x)$ we have $\vec\psi_\e \in C^\infty(K, \cD_L(U))$ for $\e$ small enough, which gives
\[ \norm{R^s_\e}_{K,m} \le \lambda ( \norm{ \vec\psi_\e }_{K,c; L, l} ). \]
Consequently, $(R^s_\e)_\e \in \cE^s_M(\Omega)$ follows from
\[ \norm{\vec\psi_\e}_{K,c; L, l} = \sup_{x, \alpha, y, \beta} \abso{ \pd_x^\alpha \pd_y^\beta \bigl( \rho_\e(x-y) \chi ( ( x-y) \abso{\ln \e}) \kappa_\e(y)\bigr) } = O(\e^{-n-c-l}). \]
For negligibility we proceed similarly; the claim then follows by using that for a bounded subset $B \subseteq C^\infty(U)$ we have $\norm{\vec\psi_\e - \vec\delta}_{K,c; B} = O(\e^N)$ for all $N \in \bN$, which is seen as in \cite[Prop.~12, p.~38]{zbMATH06172905} and actually merely a restatement of the fact that $\iota^s f - \sigma^s f = O(\e^N)$ for all $N$ uniformly for $f \in B$.
\end{proof}

\subsection{The diffeomorphism invariant algebra}

There are several variants of the diffeomorphism invariant algebra $\cG^d$; we will employ the following formulation \cite{papernew, bigone, specfull}:
\begin{align*}
 \cE^d(\Omega) &\coleq C^\infty( \cD(\Omega), C^\infty(\Omega)) \\
 \cE_M^d(\Omega) &\coleq \{ R \in C^\infty(\cD(\Omega))\ |\ \forall K \csub \Omega\ \forall k,m \in \bN_0\ \forall (\vec\varphi_\e)_\e \in S(\Omega)\ \forall (\vec\psi_{1,\e})_\e, \dotsc,\\
& \qquad (\vec\psi_{k,\e})_\e \in S^0(\Omega)\ \exists N \in \bN: \norm{ \ud^k R(\vec\varphi_\e)(\vec\psi_{1,\e}, \dotsc, \vec\psi_{k,\e} )}_{K,m} = O(\e^{-N} ) \}, \\
 \cN^d(\Omega) & \coleq \{ R \in C^\infty(\cD(\Omega))\ |\ \forall K \csub \Omega\ \forall k,m \in \bN_0\ \forall (\vec\varphi_\e)_\e \in S(\Omega)\ \forall (\vec\psi_{1,\e})_\e, \dotsc,\\
& \qquad (\vec\psi_{k,\e})_\e \in S^0(\Omega)\ \forall N \in \bN: \norm{ \ud^k R(\vec\varphi_\e)(\vec\psi_{1,\e}, \dotsc, \vec\psi_{k,\e} )}_{K,m} = O(\e^N ) \}, \\
 \cG^d(\Omega) & \coleq \cE_M^d(\Omega) / \cN^d(\Omega), \\
 (\iota^d u) (\varphi)(x) & \coleq \langle u, \varphi \rangle, \\
 (\sigma^d f)(\varphi)(x) & \coleq f(x).
\end{align*}

The spaces $S(\Omega)$ and $S^0(\Omega)$ employed in this definition are given as follows:

\begin{definition}\label{def_testobj} Let a net of smoothing kernels $(\vec\varphi_\e)_\e \in C^\infty(\Omega, \cD(\Omega))^I$ be given and denote the corresponding net of smoothing operators by $(\Phi_\e)_\e \in \cL ( \cD'(\Omega), C^\infty(\Omega))^I$. Then $(\varphi_\e)_\e$ is called a \emph{test object} on $\Omega$ if
 \begin{enumerate}[label=(\roman*)]
  \item\label{def_testobj.1} $\Phi_\e \to \id$ in $\cL(\cD'(\Omega), \cD'(\Omega))$,
  \item\label{def_testobj.2} $\forall p \in \csn ( \cL(\cD'(\Omega), C^\infty(\Omega)) )$ $\exists N \in \bN$: $p ( \Phi_\e ) = O (\e^{-N})$,
  \item\label{def_testobj.3} $\forall p \in \csn ( \cL(C^\infty(\Omega), C^\infty(\Omega)))$ $\forall m \in \bN$: $p( \Phi_\e|_{C^\infty(\Omega)} - \id ) = O(\e^m)$,
  \item\label{def_testobj.4} $\forall x \in \Omega$ $\exists V \in \cU_x(\Omega)$ $\forall r>0$ $\exists \e_0>0$ $\forall y \in V$ $\forall \e < \e_0$: $\supp \varphi_\e(y) \subseteq B_r(y)$.
 \end{enumerate}
We denote the set of test objects on $\Omega$ by $S(\Omega)$. Similarly, $(\vec\varphi_\e)_\e$ is called a $0$-test object if it satisfies these conditions with \ref{def_testobj.1} and \ref{def_testobj.3} replaced by the following conditions:
\begin{enumerate}[label=(\roman*')]
 \item\label{def_testobj.1p}  $\Phi_\e \to 0$ in $\cL(\cD'(\Omega), \cD'(\Omega))$,\addtocounter{enumi}{1}
 \item\label{def_testobj.3p} $\forall p \in \csn ( \cL(C^\infty(\Omega), C^\infty(\Omega)))$ $\forall m \in \bN$: $p( \Phi_\e|_{C^\infty(\Omega)} ) = O(\e^m)$.
\end{enumerate}
The set of all $0$-test objects on $\Omega$ is denoted by $S^0(\Omega)$.
\end{definition} 

\begin{definition}
 For $R \in \cB(\Omega)$ we define $R^d \in \cE^d(\Omega)$ by
 \[ R^d(\varphi)(x) \coleq R ( [x' \mapsto \varphi ] ) (x).  \]
\end{definition}

\begin{proposition}\label{37}
 \begin{enumerate}[label=(\roman*)]
  \item\label{37.1} $(\iota u)^d = \iota^d u$ for $u \in \cD'(\Omega)$.
  \item\label{37.2} $(\sigma f)^d = \sigma^d u$ for $f \in C^\infty(\Omega)$.
  \item\label{37.3} $R^d \in \cE^d_M(\Omega)$ for $R \in \cM(\Omega)$.
  \item\label{37.4} $R^d \in \cN^d(\Omega)$ for $R \in \cN(\Omega)$.
 \end{enumerate}
\end{proposition}
\begin{proof}
 \ref{37.1} and \ref{37.2} are clear from the definition. \ref{37.3} and \ref{37.4} follow directly from the estimates
\begin{gather*}
 \norm{\vec\varphi_\e}_{K,c; L, l} = O(\e^{-N})\qquad \textrm{ for some }N,\\
 \norm{\vec\varphi_\e - \vec\delta}_{K,c; B} = O(\e^N)\qquad \textrm{ for all }N,
\end{gather*}
which hold by definition of the spaces $S(\Omega)$ and $S^0(\Omega)$.
\end{proof}

\subsection{The elementary algebra}


For Colombeau's elementary algebra we employ the formulation of \cite[Section 1.4]{GKOS}, Section 1.4. For $k \in \bN_0$ we let $\cA_k(\bR^n)$ be the set of all $\varphi \in \cD(\bR^n)$ with integral one such that, if $k \ge 1$, all moments of $\varphi$ order up to $k$ vanish.
\begin{align*}
 U^e(\Omega) & \coleq \{ (\varphi, x) \in \cA_0(\bR^n) \times \Omega\ |\ x + \supp \varphi \subseteq \Omega \} \\
 \cE^e(\Omega) & \coleq \{ R \colon U^e(\Omega) \to \bC\ |\ \forall \varphi \in \cA_0(\bR^n): R(\varphi,.)\textrm{ is smooth} \} \\
\cE^e_M(\Omega) & \coleq \{ R \in \cE^e(\Omega) \ |\ \forall K \csub \Omega\ \forall m \in \bN_0\ \exists N \in \bN\ \forall \varphi \in \cA_N(\bR^n): \\
\qquad\qquad& \norm{ R(S_\e\varphi, .) }_{K,m} = O(\e^{-N}) \} \\
\cN^e(\Omega) & \coleq \{ R \in \cE^e(\Omega) \ |\ \forall K \csub \Omega\ \forall m \in \bN_0\ \forall N \in \bN\ \exists q \in \bN\ \forall \varphi \in \cA_q(\bR^n): \\
\qquad\qquad& \norm{ R(S_\e\varphi, .) }_{K,m} = O(\e^N) \} \\
\cG^e(\Omega) &\coleq \cE^e_M(\Omega) / \cN^e(\Omega) \\
(\iota^e u)(\varphi,x) & \coleq \langle u, \varphi(.-x) \rangle \\
(\sigma^e f)(\varphi,x) & \coleq f(x)
\end{align*}

\begin{definition}For $R \in \basice(\Omega)$ we define $R^e \in \cE^e(\Omega)$ by $R^e(\varphi, x) \coleq R(\varphi, x)$.
\end{definition}

\begin{proposition}\label{38}
 \begin{enumerate}[label=(\roman*)]
  \item\label{38.1} $(\iotae u)^e = \iota^e u$ for $u \in \cD'(\Omega)$.
  \item\label{38.2} $(\sigmae f)^e = \sigma^e u$ for $f \in C^\infty(\Omega)$.
  \item\label{38.3} $R^e \in \cE^e_M(\Omega)$ for $R \in \mode(\Omega)$.
  \item\label{38.4} $R^e \in \cN^e(\Omega)$ for $R \in \nege(\Omega)$.
 \end{enumerate}
\end{proposition}
\begin{proof}
 Again, \ref{38.1} and \ref{38.2} are clear from the definition. For \ref{38.3}, fix $K \csub \Omega$ and $m \in \bN_0$. From \Cref{emodnegchar} we obtain $r$, $c$ and $\lambda$ such that for $\supp \varphi \subseteq B_r(0)$, $\norm{R(\varphi, .)}_{K,m} \le \lambda ( \norm{\varphi}_c)$.
For $\varphi \in \cA_0(\bR^n)$ and $\e$ small enough, $\supp S_\e\varphi \subseteq B_r(0)$, so we only have to take into account that $\norm{S_\e\varphi}_c = O(\e^{-N})$ for some $N \in \bN$. Similarly, \ref{38.4} is obtained from the fact that given any $N$, for $q$ large enough we have $\norm{ (S_\e\varphi)^* - \vec\delta}_{K,c; B} = O(\e^N)$ for all $\varphi \in \cA_q(\bR^n)$.
\end{proof}

\textbf{Acknowledgments.} This research was supported by project P26859-N25 of the Austrian Science Fund (FWF).

\printbibliography

\end{document}